\documentclass[pdflatex,sn-mathphys-num]{sn-jnl}

\usepackage{graphicx}%
\usepackage{multirow}%
\usepackage{amsmath,amssymb,amsfonts}%
\usepackage{amsthm}%
\usepackage{mathrsfs}%
\usepackage[title]{appendix}%
\usepackage{xcolor}%
\usepackage{textcomp}%
\usepackage{manyfoot}%
\usepackage{booktabs}%
\usepackage{algorithm}%
\usepackage{algorithmicx}%
\usepackage{algpseudocode}%
\usepackage{listings}%

\theoremstyle{thmstyleone}%
\newtheorem{theorem}{Theorem}
\newtheorem{lemma}[theorem]{Lemma}
\newtheorem{corollary}[theorem]{Corollary}

\theoremstyle{thmstyletwo}%

\theoremstyle{thmstylethree}%
\newtheorem{definition}{Definition}%

\newcommand{\bbP}{\mathbb{P}}
\newcommand{\cE}{\mathcal{E}}
\newcommand{\cI}{\mathcal{I}}
\newcommand{\cV}{\mathcal{V}}
\newcommand*\dd{\mathop{}\!\mathrm{d}}
\newcommand{\E}[1]{$\times10^{#1}$}
\newcommand{\Erellow}{\mathcal{E}_\text{rel}^\text{low}}
\newcommand{\Erelup}{\mathcal{E}_\text{rel}^\text{up}}
\newcommand{\R}{\mathbb{R}}

\begin{document}

\title[Two-sided eigenvalue bounds for the Euler--Bernoulli beam]{Two-sided eigenvalue bounds for the Euler--Bernoulli beam}


\author[1]{\fnm{Jana} \sur{Burkotov\'a}}\email{jana.burkotova@upol.cz}
\equalcont{These authors contributed equally to this work.}

\author[1]{\fnm{Jitka} \sur{Machalov\'a}}\email{jitka.machalova@upol.cz}
\equalcont{These authors contributed equally to this work.}

\author*[2]{\fnm{Tom\'a\v{s}} \sur{Vejchodsk\'y}}\email{vejchod@math.cas.cz}
\equalcont{These authors contributed equally to this work.}

\affil[1]{\orgdiv{Faculty of Science}, \orgname{Palacký University Olomouc}, \orgaddress{\street{17. listopadu 12}, \city{Olomouc}, \postcode{CZ-77900}, \country{Czech Republic}}}

\affil*[2]{\orgdiv{Institute of Mathematics}, \orgname{Czech Academy of Sciences}, \orgaddress{\street{\v{Z}itn\'a 25}, \city{Praha}, \postcode{CZ-11567}, \country{Czech Republic}}}


\abstract{%
We derive novel guaranteed lower bounds for eigenvalues of the Euler--Bernoulli beam with variable bending stiffness. While the standard finite element Rayleigh--Ritz method automatically yields upper bounds, we obtain lower bounds by employing interpolation error estimates with the explicitly known value of the associated constant. This approach is especially efficient and easy to apply for piecewise constant bending stiffness. For general variable material parameters, we obtain guaranteed lower bounds through an auxiliary beam-bending problem. The first eigenvalue is of primary interest in applications because it represents the critical load that causes buckling of the beam. Our method is, however, suitable also for the higher buckling modes. In addition, it can be applied to the physically more relevant nonlinear Gao beam model with piecewise constant bending stiffness, which has the same first eigenvalue as the classical Euler--Bernoulli beam. The presented numerical experiments illustrate the performance of the proposed eigenvalue bounds, demonstrating their convergence rates.
}




\maketitle

\section{Introduction}
Buckling of elastic columns and beams is a classical problem in structural mechanics with significant practical importance. 
The critical load at which a straight beam becomes unstable and buckles is determined by the smallest eigenvalue of the governing differential equation for bending \cite{Timoshenko1961}. 
	
For an elastic Euler--Bernoulli beam, the buckling eigenvalue problem can be formulated as finding the axial compressive force $P$ (eigenvalue) for which a nontrivial lateral deflection shape exists. 
Accurate determination of this critical load is crucial to ensure safety and reliability in engineering design. 
The higher eigenvalues corresponding to higher buckling modes can be of interest in specialized applications.

The Euler--Bernoulli beam buckling problem, being a fourth-order ODE, can be solved analytically for uniform material properties. In the case of a clamped-clamped beam, the smallest eigenvalue corresponds to a critical load $P_1 = 4\pi^2 EI/L^2$ for a uniform beam of length $L$ and bending stiffness $EI$, see for example \cite{Bazant2010}. 
However, for non-uniform beams, one resorts to numerical methods.

The Euler--Bernoulli buckling problem is typically discretized by finite elements and solved as a generalized eigenvalue problem. 
The finite element method, as a special case of the Rayleigh--Ritz method, is known to produce optimally convergent approximations of eigenvalues \cite{BabOsb:1991}, but it inherently approximates the exact eigenvalues from above, see e.g. \cite{StrFix2008}.
In particular, the smallest computed eigenvalue $P_{h,1}$ will always overestimate the true critical load $P_1$.
Therefore, we are interested in the computation of a \emph{lower bound} on $P_1$. Only a lower bound can guarantee a limiting load under which the buckling does not occur. 
Thus, lower bounds are necessary for an efficient and safe design.
Computing both lower and upper bounds provides us with a guaranteed bound on the error $|P_1-P_{h,1}|$. This error quantifies how much the computed critical load $P_{h,1}$ overestimates the true value $P_1$. Thus, it enables engineers to make decisions based on reliable information about the accuracy of computed approximations.

However, computing rigorous lower bounds for eigenvalues of differential operators is a challenging problem. 
Over the years, various methods have been developed to estimate or bound eigenvalues from below. 
Classical approaches include the use of comparison theorems and domain monotonicity \cite{Plum1990,Plum1991}, the Lehmann--Goerisch method \cite{BehMerPluWie2000,GoeHau1985,Lehmann1949,Lehmann1950}, as well as verified computing techniques \cite{Liu2015,LiuOis2013}. 
In recent works, finite element-based \textit{a posteriori} error estimation complemented with flux reconstruction techniques has been employed to derive guaranteed two-sided bounds for eigenvalues of self-adjoint operators \cite{CanDusMadStaVoh2017,CanDusMadStaVoh2018,CanDusMadStaVoh2020}. 
A comparison of modern methods for two-sided eigenvalue bounds for second-order eigenvalue problems can be found in \cite{Vejchodsky2018b}, and a study on flux reconstruction in the Lehmann--Goerisch method in \cite{Vejchodsky2018}.

Our current approach is based on the method presented in \cite{CarGal2014,CarGed2014,Liu2015,LiuOis2013}.
These works use nonconforming elements such as Cruzeix--Raviart and Morley finite elements. These elements possess a special property: the corresponding energy projector equals the interpolation operator. Utilizing the explicitly known constant in the interpolation error estimate, they derive computable lower bounds for eigenvalues of the Laplace and biharmonic operators.

Despite these advances, relatively few works have focused specifically on the Euler--Bernoulli beam eigenvalue problem for buckling in the context of guaranteed bounds. To our knowledge, guaranteed eigenvalue bounds for these beam problems are new. 

The Euler--Bernoulli beam is described by a fourth-order differential operator, and it is of great interest in mechanical engineering. 
Papers \cite{RadMachBur2020,RadMach2023,RadMach2024} investigate identification problems for nonlinear Gao beam \cite{Gao1996} and, in the process, derive certain inequalities for beam functions under various boundary conditions. In particular, the Wirtinger-type inequality, which is crucial for our considerations.

For post-buckling analysis, the nonlinear Gao beam model provides a more adequate framework than the classical Euler–Bernoulli formulation \cite{NetMach2023}. In that work, the first buckling eigenvalue of the Gao beam was analyzed and shown to coincide with the corresponding eigenvalue of the classical Euler–Bernoulli beam 
provided the bending stiffness is piecewise constant.
Consequently, any rigorous bounds on the first critical load obtained for the Euler–Bernoulli problem can be directly transferred to the Gao beam model. In addition, unlike the Euler–Bernoulli model, which yields only admissible buckling mode shapes, the Gao beam model guarantees the existence of three distinct solutions associated with different post-buckling states, thus offering a substantially richer description of the buckling behaviour. Let us note that the relation between the higher eigenvalues of Euler-Bernoulli and Gao beams remains unexplored.

In this paper, we develop a simple and effective procedure to obtain two-sided bounds for the eigenvalues of a clamped Euler--Bernoulli beam with both piecewise constant and varying bending stiffness.
Our approach uses the conforming cubic Hermite finite elements to compute upper bounds on eigenvalues. 
Lower bounds are obtained by an easy-to-evaluate formula involving an interpolation constant. 
In the case of piecewise constant bending stiffness, we explicitly determine the value of this constant and obtain effective eigenvalue bounds virtually without any computational overhead beyond the standard finite element solution.
Furthermore, we demonstrate how to generalize these bounds to the case of generally varying material properties. 
This generalization can be seen as our main theoretical contribution.

The remaining part of this paper is organized as follows. 
Section~\ref{se:beam} introduces the mathematical formulation of the Euler--Bernoulli beam eigenvalue problem, including its discretization by cubic Hermite finite elements. 
In Section~\ref{se:lowerbounds}, we review the theoretical concept of lower bounds on eigenvalues, define the Hermite interpolation operator and prove its important property.
Section~\ref{se:pwconst} proves guaranteed lower bounds on eigenvalues of the Euler--Bernoulli beam with piecewise constant bending stiffness.
Section~\ref{eq:variablestiffness} generalizes these bounds to the case of variable bending stiffness.
Section~\ref{se:steppedbeam} applies derived bounds to the special case of the stepped beam with rectangular cross-section, constant Young's modulus, and piecewise constant thickness,
and comments on the applicability to the nonlinear Gao beam.
Section~\ref{eq:numex} contains numerical experiments that validate the theoretical bounds and present their orders of convergence for beams of rectangular and circular cross-sections with uniform and variable thickness.
Finally, Section~\ref{eq:conclusion} concludes the paper and suggests possible directions for future research.

\section{The beam eigenvalue problem and its numerical approximation}
\label{se:beam}

We consider an elastic Euler--Bernoulli beam of length $L$, with varying bending stiffness $E(x)I(x)$, $x \in (0, L)$, where $E = E(x) > 0$ is the Young's modulus and $I = I(x) > 0$ is the second moment of area of the cross-section. 
We assume that the product $E(x)I(x)$ is bounded in the sense that there exist constants $C_1 > 0$ and $C_2 > 0$ such that
\begin{equation}
  \label{eq:EIbounded}
  0 < C_1 \leq E(x)I(x) \leq C_2 \quad \forall x \in (0,L).
\end{equation}

The beam is assumed to be clamped at both ends ($x=0$ and $x=L$) and is subjected to an axial compressive force $P$ applied at the end point $x=L$.
Denoting by prime the derivative with respect to $x$, the governing differential equation for the deflection $u(x)$ (transverse displacement) in the buckling configuration is:
\begin{equation}
  \label{eq:strong-form}
  [E(x)I(x)\,u''(x)]'' + P\,u''(x) \;=\; 0, \qquad 0 < x < L,
\end{equation}
with clamped boundary conditions 
$$
  u(0)=u'(0) = 0, \qquad u(L)=u'(L) = 0.
$$ 
Equation \eqref{eq:strong-form} is a fourth-order ordinary differential equation. Physically, the term $(EI u'')''$ represents the bending resistance of the beam, while $P\,u''$ represents the effect of axial compression (often called the geometric stiffness term). 
Nontrivial solutions $u(x)$ exist only for specific values of $P$, which are the critical loads leading to buckling. These values correspond to the eigenvalues of the problem. The smallest such $P$ is the critical buckling load of the beam.
		
To formulate the problem in the variational form, we introduce the appropriate function space
$$
V = H^2_0(0,L) = \{ v \in H^2(0,L) : v(0)=v'(0)=v(L)=v'(L)=0 \}.
$$ 
It is the Sobolev space of functions with square integrable second generalized derivatives that satisfy the clamped boundary conditions.
Further, we introduce bilinear forms $a(u,v)$ and $b(u,v)$ on $V \times V$ as
$$
	a(u,v) = \int_0^L E(x) I(x)\, u''(x) v''(x)\dd x
	\quad\text{and}\quad
	b(u,v) = \int_0^L u'(x) v'(x) \dd x.
$$
Note that both these forms are symmetric and positive definite on $V$, and induce the following norms
$$
\|v\|_a^2 = a(v,v) = \int_0^L EI\, (v'')^2 \dd x
\quad\text{and}\quad 
\|v\|_b^2 = b(v,v) = \int_0^L (v')^2 \dd x
\quad \forall v \in V,
$$
where we do not explicitly indicate the dependence of $E$, $I$, and $v$ on $x$.

The weak (variational) form of the eigenvalue problem reads: 
find eigenvalues $P_i \in \mathbb{R}$, $i=1,2,\dots$, and a nonzero eigenfunctions $u_i \in V$ such that 
\begin{equation}
    \label{weak-form}
	a(u_i,v) = P_i\, b(u_i,v) \qquad \forall v \in V.
\end{equation}
This eigenvalue problem is well posed due to the spectral theory of compact operators. It possesses a countable sequence of positive eigenvalues $ 0 < P_1 \leq P_2 \leq \cdots$ converging towards infinity.
The critical buckling load given by $P_1$ is of special interest.

To solve problem \eqref{weak-form} approximately, we employ the usual conforming finite element method. Let us consider a finite element partition $0=x_1 < x_2 < \cdots < x_N < x_{N+1} = L$
defining $N$ elements $K_k = [x_k,x_{k+1}]$, $k=1,2,\dots,N$. 
The length of each element is denoted by $h_k = x_{k+1} - x_k$, $k=1,2,\dots,N$, and we set $h = \max_{k=1,2,\dots,N} h_k$.
Further, we consider the cubic Hermite finite element space 
$$
  V_h = \{ v_h \in V : v_h|_{K_k} \in \bbP^3(K_k) \text{ for all } k=1,2,\dots,N\},
$$
where $\bbP^3(K_k)$ denotes the space of cubic polynomials on the interval $K_k$.
Note that the requirement $v_h \in V$ implies that $V_h \subset C^1([0,L])$.
Approximate eigenvalues $P_{h,i} \in \mathbb{R}$, $i=1,2,\dots,\operatorname{dim} V_h$, and the corresponding approximate eigenfunctions $u_{h,i} \in V_h \setminus \{0\}$ are determined by the identity
\begin{equation}
	\label{eq:fem}
	a(u_{h,i},v_h) = P_{h,i}\, b(u_h,v_h) \qquad \forall v_h \in V_h.
\end{equation}
It is well known that by the Courant–Fischer–Weyl min-max principle, the discrete values $P_{h,i}$ approximate the exact eigenvalues $P_i$ from above \cite{BabOsb:1991,StrFix2008}.

\section{Theoretical lower bounds and Hermit interpolation}	
\label{se:lowerbounds}

In this section, we first use theoretical results presented in \cite{CarGal2014,CarGed2014,Liu2015,LiuOis2013}
to derive lower bounds on eigenvalues $P_i$ of problem \eqref{weak-form}.
For the reader's convenience, we formulate here the crucial result of \cite[Theorem~2.1]{Liu2015} using our notation.

\begin{definition}
\label{de:energyproj}
For any $u \in V$, define the energy projection $\Pi_h u \in V_h$ as the unique function satisfying
$$
  a(u - \Pi_h u,v_h) = 0 \quad \forall v_h \in V_h.
$$
The linear operator $\Pi_h : V \rightarrow V_h$ is called the energy projector.
\end{definition}

\begin{theorem} 
\label{th:lowerbound}
Let the exact eigenvalues $P_i$ defined by \eqref{weak-form} be approximated by $P_{h,i}$ given in \eqref{eq:fem}.
Let the energy projector $\Pi_h$ satisfy the inequality
\begin{equation}
  \label{eq:eprojestim}
  \| u - \Pi_h u \|_b \leq C_h \| u - \Pi_h u \|_a
  \quad \forall u \in V,
\end{equation}
where $C_h > 0$ depends on the mesh size $h$ in general.
Then,
$$
  \frac{P_{h,i}}{1 + P_{h,i} C_h^2} \leq P_i  
  \quad \text{for all } i = 1,2,\dots,\operatorname{dim}V_h.
$$
\end{theorem}

Further, we define the Hermite interpolation operator $\cI_h : V \rightarrow V_h$ and prove its important property. Note that the Sobolev embedding theorem yields $V = H^2_0(0,L) \subset C^1([0,L])$. In particular, classical derivatives $u'$ of any $u \in V$ are well defined in all points of $(0,L)$ as well as the corresponding one-sided derivatives at the end-points $0$ and $L$.

\begin{definition}
\label{de:Ih}
The Hermite interpolation operator $\cI_h : V \rightarrow V_h$ maps any $u \in V$
to a function $\cI_h u \in V_h$ which is determined by requirements
$$
  \cI_h u(x_i) = u(x_i)
  \quad\text{and}\quad 
  (\cI_h u)'(x_i) = u'(x_i)
$$  
for all nodes $x_i$, $i = 1,2,\dots,N+1$, of the finite element partition.
\end{definition}

Note that the requirement $(\cI_h u)'(x_i) = u'(x_i)$ has to be understood in the sense of one-sided derivatives of $\cI_h u$. 
For interior nodes $x_i$, $i=2,3,\dots,N$, we require both one-sided derivatives of $\cI_h u$ at point $x_i$ to be equal to the derivative $u'(x_i)$. For boundary nodes $x_1 = 0$ and $x_{N+1} = L$, this requirement means $(\cI_h u)'_+(x_1) = u'_+(x_1)$ and $(\cI_h u)'_-(x_{N+1}) = u'_-(x_{N+1})$, where the subindices $\pm$ stand for the derivative from the right and left, respectively.
Since the space $V_h$ consists of piecewise cubic polynomials, conditions in Definition~\ref{de:Ih} determine the function $\cI_h u$ uniquely.

To formulate the following lemma, we define constants $\kappa_k$ as minimal values of the bending stiffness $E(x)I(x)$ over elements $K_k$, respectively, namely
\begin{equation}
\label{eq:kappa}
\kappa_k = \min_{x \in K_k}{E(x)I(x)} 
\quad\text{for all }k=1,2,\dots,N,
\end{equation}
and introduce the mesh size scaled by the bending stiffness
\begin{equation}
  \label{eq:hEI}
  h_{EI}^2 = \max_{k=1,2,\dots,N} \frac{h_k^2}{\kappa_k},
\end{equation}
where we recall that $h_k$ stands for the length of $K_k$.
Using the scaled mesh size $h_{EI}$, we present the crucial inequality for the error of the interpolation operator $\cI_h$. 

\begin{lemma}
\label{le:interpestim}
The Hermite interpolation operator $\cI_h$ satisfies the estimate
\begin{equation*}
  \|u - \cI_h u\|_b \le \frac{h_{EI}}{2\pi} \|u - \cI_h u\|_a
  \quad\forall u \in V.
\end{equation*}
\end{lemma}
\begin{proof}
Let us recall Wirtinger's inequality
\begin{equation}
\label{eq:Almansi}
\int_\alpha^\beta \big(f(x) \big)^2 \dd x \ \leq\ \frac{(\beta - \alpha)^2}{4\pi^2} \int_\alpha^\beta \big( f'(x) \big)^2 \dd x
\end{equation}
valid for all functions $f \in H^1_0(\alpha,\beta)$ such that $\int_{\alpha}^{\beta} f(x) \dd x=0$, see \cite{Mitr1970}. Note that the Sobolev space $H^1_0(\alpha,\beta)$ consists of all $f \in L^2(\alpha,\beta)$ such that their distributional derivative lies in $L^2(\alpha,\beta)$ and $f(\alpha) = f(\beta) = 0$.

Now, let us consider an arbitrary $u \in V$ and set $w = u - \cI_h u$. 
Notice that $w|_{K_k} \in H^2_0(K_k)$ and consequently $w'|_{K_k} \in H^1_0(K_k)$ for all elements $K_k = [x_k, x_{k+1}]$, $k=1,2,\dots,N$. 
Since $\int_{x_k}^{x_{k+1}} w'(x) \dd x = w(x_{k+1}) - w(x_k) = 0$,
we can apply the Wirtinger's inequality \eqref{eq:Almansi} to $w'(x)$ on all intervals $K_k$ and obtain
\begin{multline*}
  \| w \|_b^2 = \int_0^L \left( w' \right)^2 \dd x 
  = \sum_{k=1}^N \int_{x_k}^{x_{k+1}} \left( w' \right)^2 \dd x 
  \leq \sum_{k=1}^N \frac{h_k^2}{4\pi^2 \kappa_k} \int_{x_k}^{x_{k+1}} \kappa_k \left( w'' \right)^2 \dd x 
\\
  \leq  
  \frac{1}{4\pi^2} \left(\max_{k=1,2,\dots,N} \frac{h_k^2}{\kappa_k}\right) \sum_{k=1}^N \int_{x_k}^{x_{k+1}} E(x)I(x) \left( w'' \right)^2 \dd x 
  = \frac{h_{EI}^2}{4\pi^2} \| w \|_a^2.
\end{multline*}
Note that we again do not explicitly indicate the dependence of $w$ on $x$.
\end{proof}

\section{Piecewise constant bending stiffness}
\label{se:pwconst}

Within this section, we consider the bending stiffness $E(x)I(x)$ to be piecewise constant.
More precisely, we assume partition $0=y_1 < y_2 < \dots y_r < y_{r+1}=L$ of the interval $[0,L]$ into $r$ mutually disjoint segments $I_j = (y_j,y_{j+1})$, $j=1,2,\dots,r$,
and a positive constant value of the bending stiffness $E(x)I(x)$ on each segment $I_j$, $j=1,2,\dots,r$. 
Further, we assume that the finite element partition $0=x_1 < x_2 < \cdots < x_N < x_{N+1} = L$ is aligned with the partition $0=y_1 < y_2 < \dots y_r < y_{r+1}=L$ in the sense that for all $j=1,2,\dots,r+1$ exists $k \in \{1,2,\dots,N+1\}$ such that $y_j = x_k$.
In short, we say that the finite element partition aligns with the piecewise constant bending stiffness $E(x)I(x)$.

\begin{lemma}
\label{le:auIhu}
Let the bending stiffness $E(x)I(x)$ be piecewise constant and let the finite element partition align with it. 
Given any $u \in V$, the Hermite interpolation operator $\cI_h$ satisfies the identity
$$
  a(u - \cI_h u, v_h) = 0 \quad \forall v_h \in V_h.
$$
\end{lemma}
\begin{proof}
Notice that the value $\kappa_k$ given by \eqref{eq:kappa} trivially equals the
constant value of the bending stiffness $E(x)I(x)$ on the element $K_k$ for all $k=1,2,\dots,N$.
Consider any $u\in V$ and $v_h \in V_h$.
Denoting $w = u - \cI_h u$ for brevity, we use integration by parts to obtain the identity
\begin{multline*}
	a(w, v_h) = \int_0^L E(x)I(x)\,w'' v_h'' \dd x 
    = \sum_{k=1}^N \kappa_k \int_{x_k}^{x_{k+1}} w'' v_h'' \dd x 
  \\
	= -\sum_{k=1}^N \kappa_k \int_{x_k}^{x_{k+1}} w' v_h''' \dd x 
	= \sum_{k=1}^N \kappa_k  \int_{x_k}^{x_{k+1}} w v_h'''' \dd x = 0.
\end{multline*}	
Boundary terms in the integration by parts vanish due to the definition of $\cI_h$, and the last equality holds since $v_h$ is a cubic polynomial in every element $K_k = [x_k,x_{k+1}]$.
\end{proof}

Comparing the statement of this lemma with Definition~\ref{de:energyproj}, we come to an important conclusion.

\begin{corollary}
\label{eq:Pih=Ih}
If the bending stiffness $E(x)I(x)$ is piecewise constant and if the finite element partition aligns with it, then the energy projector is equal to the Hermite interpolation operator, i.e,
$$
  \Pi_h = \cI_h.
$$
\end{corollary}

\begin{theorem}
\label{th:two-sided}
Let the bending stiffness $E(x)I(x)$ be piecewise constant and let the finite element partition align with it. 
Then the eigenvalues $P_i$ given by \eqref{weak-form} are bounded by the approximate eigenvalues $P_{h,i}$ defined in \eqref{eq:fem} as follows
\begin{equation}
  \label{eq:lowerbound}
  \frac{P_{h,i}}{1 + P_{h,i} C_h^2} \leq P_i \leq P_{h,i}
  \quad \text{for all } i = 1,2,\dots,\operatorname{dim}V_h,    
\end{equation}
where
\begin{equation}\label{eq:Ch}
  C_h = \frac{h_{EI}}{2\pi} \approx 0.1592 h_{EI}
  \quad\text{and } h_{EI} \text{ is given by \eqref{eq:hEI}}.
\end{equation}
\end{theorem}
\begin{proof}
The upper bound follows from the Courant--Fischer--Weyl min-max principle, and it is a well-known property of the Ritz--Galerkin method applied to eigenvalue problems.
The lower bound comes from Theorem~\ref{th:lowerbound}. We use the fact that $\Pi_h = \cI_h$, see Corollary~\ref{eq:Pih=Ih}, and verify assumption~\eqref{eq:eprojestim} by Lemma~\ref{le:interpestim}.
\end{proof}

The lower bound in \eqref{eq:lowerbound} is given by a simple and explicit formula in terms of the approximate eigenvalues $P_{h,i}$ and the quantity $h_{EI}$ is easily computable from the constant values of the bending stiffness $E(x)I(x)$ and element lengths.
The assumption of piecewise constant bending stiffness is crucial here. The following section shows how to generalize this result to the case of generally varying bending stiffness $E(x)I(x)$.

\section{Variable bending stiffness}
\label{eq:variablestiffness}

To derive lower bounds on eigenvalues for a beam with variable bending stiffness, we first introduce an auxiliary buckling problem. This auxiliary problem has piecewise constant bending stiffness $\kappa_k$, $k=1,2,\dots,N$, see \eqref{eq:kappa}.
Formally, we assemble values $\kappa_k$ into a piecewise constant function $\kappa = \kappa(x)$, $x \in [0,L]$, such that $\kappa(x) = \kappa_k$ for all $x \in K_k$, $k=1,2,\dots,N$.

The auxiliary buckling problem is simply defined as problem \eqref{weak-form} with the bending stiffness $E(x)I(x)$ replaced by the piecewise constant function $\kappa(x)$.
The corresponding weak formulation of this problem reads: find eigenvalues $\widetilde P_i \in \R$ and eigenfunctions $\tilde u_i \in V\setminus\{0\}$ such that
\begin{equation}
  \label{eq:auxweak}
  \tilde a(\tilde u_i,v) = \widetilde P_i b(\tilde u_i,v) \quad \forall v \in V,
\end{equation}
where
$$
  \tilde a(u,v) = \int_0^L \kappa(x) u''(x) v''(x) \dd x.
$$
Similarly to above, we consider the finite element formulation to find $\widetilde P_{h,i} > 0$ and $\tilde u_{h,i} \in V_h \setminus \{0\}$ such that
\begin{equation}
  \label{eq:auxfem}
  \tilde a(\tilde u_{h,i},v_h) = \widetilde P_{h,i} b(\tilde u_{h,i},v_h) \quad \forall v_h \in V_h.
\end{equation}

\begin{lemma}\label{le:aux}
  Let $P_i$, $i=1,2,\dots$, stand for the exact eigenvalues of the buckling problem \eqref{weak-form} with a generally varying bending stiffness $E(x)I(x)$ and let $\widetilde P_i$ be given by \eqref{eq:auxweak}. 
  Then 
  $$
    \widetilde P_i \leq P_i \quad\text{for all }i=1,2,\dots.
  $$
\end{lemma}
\begin{proof}
Eigenvalue problems \eqref{weak-form} and \eqref{eq:auxweak} are self-adjoint and correspond to the inverse of a compact operator. 
Therefore, all eigenvalues of the problem \eqref{weak-form} can be characterized by the Courant–Fischer–Weyl min-max principle as
$$
  P_i = \min_{\cE \in \cV^{(i)}} \max_{u \in \cE} \frac{\int_0^L E(x)I(x)\, [u''(x)]^2 \dd x}{\int_0^L [u'(x)]^2 \dd x},
$$
where $\cV^{(i)}$ stands for the set of all $i$-dimensional subspaces of $V$.
Since $E(x)I(x) \geq \kappa(x)$ in $(0,L)$, we immediately have
$$
  P_i \geq \min_{\cE \in \cV^{(i)}} \max_{u \in \cE} \frac{\int_0^L \kappa(x) [u''(x)]^2 \dd x}{\int_0^L [u'(x)]^2 \dd x} = \widetilde P_i,
$$
where the last equality follows from the min-max principle for the auxiliary eigenvalue problem \eqref{eq:auxweak}.
\end{proof}

\begin{corollary}\label{co:two-sided}
The eigenvalues $P_i$ given by \eqref{weak-form} are bounded by the approximate eigenvalues $P_{h,i}$ and $\widetilde P_{h,i}$ defined in \eqref{eq:fem} and \eqref{eq:auxfem} as follows
$$
  \frac{\widetilde P_{h,i}}{1 + \widetilde P_{h,i} C_h^2} \leq P_i \leq P_{h,i}
  \quad \text{for all } i = 1,2,\dots,\operatorname{dim}V_h,
$$
where $C_h$ is given by \eqref{eq:Ch}.
\end{corollary}
\begin{proof}
The upper bound is well known and follows again from the Courant--Fischer--Weyl min-max principle. The lower bound is an immediate combination of Lemma~\ref{le:aux} and Theorem~\ref{th:two-sided} applied to the auxiliary problem \eqref{eq:auxweak}.
\end{proof}

The lower bound given in Corollary~\ref{co:two-sided} is easy to compute by solving the auxiliary beam buckling problem with piecewise constant bending stiffness $\kappa(x)$. It is given by a simple formula, in terms of the auxiliary eigenvalues $\widetilde P_{h,i}$ and the interpolation error constant $C_h = h_{EI}/(2\pi)$,  see \eqref{eq:Ch}, where $h_{EI}$ can be directly found from the minimal values of the bending stiffness on the individual elements and their lengths.

\section{Application to a stepped beam}
\label{se:steppedbeam}

A stepped beam is a typical example of a practically interesting case of a beam with piecewise constant thickness. It is a special case of a beam with piecewise constant bending stiffness as discussed above in Section~\ref{se:pwconst}. 
In particular, the stepped beam has a constant Young's modulus $E$ along the beam, a rectangular cross-section of a constant width $b$, and piecewise constant thickness $t = t(x)$. 

To be precise, we consider the same segments $I_j=(y_j, y_{j+1})$, $j=1,\dots,r$, as in Section~\ref{se:pwconst} and assume that the thickness $t = t(x)$ (and hence the second area moment $I(x) = b t^3(x) / 12$) is constant on each segment $I_j$, $j=1,2,\dots,r$. 
Denoting by $t_j$ the constant thickness on segment $I_j$, we have $E\,I(x) = E\,b\, t_j^3 / 12$ for all $x \in I_j$, $j=1,2,\dots,r$.
As above, we suppose that the finite element partition $0 = x_1 < x_2 < \ldots< x_N < x_{N+1} = L$ is aligned with segments $I_k=(y_k,y_{k+1})$, $k=1,\ldots,r$.

Introducing $\lambda_i = 12 P_i/(E b)$, the eigenvalue problem \eqref{weak-form} becomes equivalent to finding $\lambda_i \in \R$, $i=1,2,\dots$, and nonzero $u_i\in V$ such that
\begin{equation}
  \label{eq:steppedproblem}
  \widehat{a}(u_i,v) = \lambda_i b(u_i,v) \quad \forall v \in V,
\end{equation}
where
\begin{equation*}
  \widehat{a}(u,v) = \int_0^L t(x)^3\, u''(x) v''(x) \dd x.
\end{equation*}
The reasonable assumption $0 < t_{\min} \le t_k \le t_{\max}$, $k=1,\dots,N$, for the thickness values ensures that condition \eqref{eq:EIbounded} is satisfied, the form $\widehat{a}$ is positive definite, and, hence, this eigenvalue problem is well posed.

We discretize problem \eqref{eq:steppedproblem} by the cubic Hermite finite element method as above. We seek approximate eigenvalues $\lambda_{h,i} \in \R$, $i=1,2,\dots\operatorname{dim}V_h$, and nonzero approximate eigenfunctions $u_{h,i} \in V_h$ such that
\begin{equation}
	\label{eq:steppedfem}
	\widehat{a}(u_{h,i},v_h) = \lambda_{h,i} b(u_{h,i},v_h) \quad \forall v_h \in V_h.
\end{equation}
Note that problems \eqref{eq:steppedproblem} and \eqref{eq:steppedfem} are special cases of \eqref{eq:auxweak} and \eqref{eq:auxfem}, respectively, with $\kappa_k = t_k^3$.
Consequently, Theorem~\ref{th:two-sided} can be easily reformulated for the specific case of the stepped beam.

\begin{corollary}
\label{co:steppedtwo-sided}
Assume a stepped beam of length $L$ with
bending stiffness $E(x)I(x) = E\,b\,t(x)^3/12$,
constant Young modulus $E$, constant width $b$, 
and piecewise constant thickness $t(x)$.
Further assume that the finite element partition $0=x_1<x_2<\cdots<x_N<x_{N+1}=L$
aligns with the piecewise constant thickness $t(x)$. 
Then the eigenvalues $\lambda_i$ of \eqref{eq:steppedproblem} 
are bounded in terms of the approximate eigenvalues $\lambda_{h,i}$ defined in \eqref{eq:steppedfem} as
\begin{equation}
  \label{eq:lowerboundstepped}
  \frac{\lambda_{h,i}}{1+\lambda_{h,i}\widehat C_h^2}
  \leq \lambda_i \leq \lambda_{h,i}
  \quad\text{for all } i = 1,2,\dots,\operatorname{dim}V_h,
\end{equation}
where
$$
\widehat C_h^2 = \frac{1}{4\pi^2}
\max_{k=1,\ldots,N}
\frac{h_k^2}{t_k^3}.
$$
\end{corollary}

As mentioned above, for the piecewise constant bending stiffness, the first eigenvalue of the Gao beam model \cite{Gao1996} is identical to the first eigenvalue of the Euler--Bernoulli beam \cite{NetMach2023}. This means that the lower bound \eqref{eq:lowerboundstepped} on $\lambda_1$ directly bounds also the critical load of the Gao beam. Differences will appear in post-buckling (nonlinear regime) analysis, which is beyond the scope of this paper. 
Nevertheless, an accurate approximation of $\lambda_1$ and its reliable lower and upper bounds may be used to initialize path-following methods for post-buckling analysis of the Gao beam.
Thus, bounds \eqref{eq:lowerboundstepped} could potentially be used in those studies to ensure a correctly bracketed initial bifurcation point.

\section{Numerical Experiments}
\label{eq:numex}

We now present numerical tests to verify the theoretical bounds derived above. 
All computations were performed using a finite element discretization with cubic Hermite elements for the clamped Euler--Bernoulli beam as described above.
We consider four cases:
(i) uniform beam with rectangular cross-section, 
(ii) stepped beam with rectangular cross-section,
(iii) uniform beam with circular cross-section,
(iv) conical beam with circular cross-section that decreases linearly along its length.

In cases (i) and (iii), the cross-section and material properties are constant and, hence, an explicit analytical expression for the eigenvalues is available. 
In contrast, in cases (ii) and (iv), the eigenvalue problem cannot be solved in closed form and numerical approximations must be employed. 

In all cases, we compute the finite element approximation of the first eigenvalue on a sequence of uniformly refined partitions and plot the corresponding convergence curves.  
All numerical experiments are conducted for the length $L = 1\,{\rm  m}$ and Young's modulus $E = 21 \times 10^{10}\, {\rm Pa}$ corresponding to steel.

In case (i), we consider a uniform beam with a rectangular cross-section of constant thickness $t = 0.015\,{\rm m}$ and width $b = 0.05\,{\rm m}$. The first eigenvalue $\lambda_1$ of \eqref{eq:steppedfem} can be computed analytically, \cite{Eisley2011}, as 
\begin{equation*}
\lambda_1 = \frac{4\pi^2 t^3}{L^2} \approx 1.33240 \times 10^{-4}\,{\rm m}
\end{equation*}
and consequently
$$
P_1 = \frac{E\,b\,\lambda_1}{12} \approx 1.165847 \times 10^5\, {\rm N}.
$$
Nevertheless, to verify the accuracy of the proposed bounds, we compute lower bound $\lambda_{h,1}^{\text{low}} = \lambda_{h,1}/(1+\lambda_{h,1}\widehat C_h^2)$ and upper bound $\lambda_{h,1}^{\text{up}} = \lambda_{h,1}$ numerically, by the cubic Hermite finite elements as in Corollary~\ref{co:steppedtwo-sided}.
Since the exact eigenvalue $\lambda_1$ is known analytically, we can compare the true relative errors $\Erellow = (\lambda_1 - \lambda_{h,1}^{\text{low}})/\lambda_1$ and $\Erelup (\lambda_{h,1}^{\text{up}} - \lambda_1)/\lambda_1$.

For numerical calculations, we use uniform partitions of the beam into $N = 2, 4, 8, 16$, $32, 64, 128$ elements,
the mesh size $h=1/N$, and the interpolation constant 
$\widehat C_h^2 = h^2/(4\pi^2t^3) \approx 7050.27h^2$.
Table~\ref{ta:uniform} shows the computed lower and upper bounds $\lambda_{h,1}^{\text{low}}$ and $\lambda_{h,1}^{\text{up}}$ on $\lambda_1$, their relative errors and experimental orders of convergence (EOC).
We observe that the lower bound on the first eigenvalue converges approximately quadratically, while the order of convergence of the upper bound approaches four. 

\begin{table}[htb]
\centering
\caption{Case (i), uniform clamped Euler--Bernoulli beam with a rectangular cross-section ($b=0.05, t= 0.015$).
Lower and upper bounds $\lambda_{h,1}^{\text{low}}$ and $\lambda_{h,1}^{\text{up}}$ of the first eigenvalue, their relative errors $\Erellow$, $\Erelup$, and experimental orders of convergence (EOC).}
\label{ta:uniform}
\begin{tabular}{lcccccc}
\toprule
$h$ & $\lambda_{h,1}^{\text{low}}$ & $\Erellow$ & EOC & $\lambda_{h,1}^{\text{up}}$ & $\Erelup$ & EOC \\
\midrule
 1/2   & 1.077154\E{-4} & 1.9157\E{-1} &       & 1.350000\E{-4} & 1.3212\E{-2} &       \\
 1/4   & 1.262895\E{-4} & 5.2163\E{-2} & 1.8767 & 1.342419\E{-4} & 7.5223\E{-3} & 0.8126  \\
 1/8   & 1.312560\E{-4} & 1.4888\E{-2} & 1.8089 & 1.333079\E{-4} & 5.1214\E{-4} & 3.8766 \\
 1/16  & 1.327255\E{-4} & 3.8585\E{-3} & 1.9480 & 1.332440\E{-4} & 3.2766\E{-5} & 3.9663 \\
 1/32  & 1.331099\E{-4} & 9.7355\E{-4} & 1.9867 & 1.332399\E{-4} & 2.0602\E{-6} & 3.9913 \\
 1/64  & 1.332072\E{-4} & 2.4395\E{-4} & 1.9967 & 1.332397\E{-4} & 1.2899\E{-7} & 3.9975 \\
 1/128 & 1.332315\E{-4} & 6.1023\E{-5} & 1.9992 & 1.332397\E{-4} & 8.3319\E{-9} & 3.9525 \\
\bottomrule
\end{tabular}
\end{table}

In case (ii), we consider a stepped beam with a rectangular cross-section. We use the same setting as for the uniform beam in case (i), except for the thickness $t$ that is now a piecewise constant function defined on eight segments of equal length
with values $0.0155$, $0.010$, $0.0153$, $0.0192$, $0.0192$, $0.053$, $0.010$, and $0.0155 \,\rm{m}$, respectively. This beam thickness is designed to be optimal in the sense of stability. It maximizes the critical buckling load by solving a shape optimization problem subject to a prescribed volume and bound constraints as described in \cite{BurMachRad2021, Haslinger2003}.

Since the exact eigenvalues are not known, we solve the eigenvalue problem \eqref{eq:steppedproblem} approximately by using the cubic Hermite elements on a sequence of uniformly refined meshes that are aligned with the piecewise constant thickness $t$. 
Table~\ref{ta:stepped8} presents the lower and upper bounds on $\lambda_1$ computed according to Corollary~\ref{co:steppedtwo-sided}. In this table, the unknown relative errors $\Erellow$ and $\Erelup$ are estimated by $\eta_\text{rel} = (\lambda_{h,1}^{\text{up}} - \lambda_{h,1}^{\text{low}})/\lambda_{h,1}^{\text{up}}$. Note that $\eta_\text{rel}$ is an upper bound on both $\Erellow$ and $\Erelup$. We clearly observe that the experimental order of convergence of $\eta_\text{rel}$ approaches two. 
It is in agreement with the results for the uniform beam, because the estimate $\eta_\text{rel}$ of the true relative errors is dominated by the error of the less accurate lower bound.

\begin{table}[htb]
\centering
\caption{Case (ii), stepped clamped beam with piecewise constant thicknesses in eight equally long segments. Lower and upper bounds $\lambda_{h,1}^{\text{low}}$ and $\lambda_{h,1}^{\text{up}}$ of the first eigenvalue,
the corresponding estimate $\eta_\text{rel}$ of the relative error, and experimental orders of convergence (EOC).}
\label{ta:stepped8}
\begin{tabular}{lcccc}
\toprule
$h$ & $\lambda_{h,1}^{\text{low}}$ & $\lambda_{h,1}^{\text{up}}$ & $\eta_\text{rel}$ & EOC \\
\midrule
1/8     &   1.501389\E{-4}  &1.596242\E{-4}  &  5.9423\E{-2} &           \\
1/16    &   1.570246\E{-4}  &1.595028\E{-4}  &  1.5537\E{-2} &   1.9353  \\
1/32    &   1.588612\E{-4}  &1.594879\E{-4}  &  3.9297\E{-3} &   1.9832  \\
1/64    &   1.593297\E{-4}  &1.594869\E{-4}  &  9.8532\E{-4} &   1.9958  \\
1/128   &   1.594475\E{-4}  &1.594868\E{-4}  &  2.4651\E{-4} &   1.9989  \\
1/256   &   1.594770\E{-4}  &1.594868\E{-4}  &  6.1639\E{-5} &   1.9997  \\
\bottomrule
\end{tabular}
\end{table}

In case (iii), we consider the uniform Euler–Bernoulli beam with circular cross-section of radius $r=0.01\,\rm{m}$.
The second moment of area about the centroidal axis is given by $I=\pi r^4/4$. 
The first eigenvalue is given by the standard analytical formula 
\begin{equation*}
\lambda_1 = \frac{4\pi^2 r^4}{L^2} \approx 3.947842 \times 10^{-7}\,{\rm m}
\end{equation*}
and thus
\begin{equation*}
P_1 = \frac{E\pi \lambda_1}{4} \approx 6.511318 \times 10^{4}\,{\rm N}.
\end{equation*}
Table~\ref{ta:circular} shows the computed lower and upper bounds $\lambda_{h,1}^{\text{low}}$ and $\lambda_{h,1}^{\text{up}}$ on $\lambda_1$, their relative errors $\Erellow$ and $\Erelup$, and experimental orders of convergence.
In agreement with case~(i), we observe the second order of convergence for the lower bound and the fourth order of convergence for the upper bound.

\begin{table}[htb]
\centering
\caption{Case (iii), uniform clamped Euler--Bernoulli beam with a circular cross-section of radius $r=0.01\,\rm{m}$.
Lower and upper bounds $\lambda_{h,1}^{\text{low}}$ and $\lambda_{h,1}^{\text{up}}$ of the first eigenvalue, their relative errors $\Erellow$ and $\Erelup$, and experimental orders of convergence (EOC).}
\label{ta:circular}
\begin{tabular}{lcccccc}
\toprule
$h$ & $\lambda_{h,1}^{\text{low}}$ & $\Erellow$ & EOC & $\lambda_{h,1}^{\text{up}}$ & $\Erelup$ & EOC \\
\midrule
 1/2   & 3.191567\E{-7} & 1.9157\E{-1} &       & 4.000000\E{-7} & 1.3212\E{-2} &       \\
 1/4   & 3.741910\E{-7} & 5.2163\E{-2} & 1.8767 & 3.977539\E{-7} & 7.5223\E{-3} & 0.8126 \\
 1/8   & 3.889066\E{-7} & 1.4888\E{-2} & 1.8089 & 3.949864\E{-7} & 5.1214\E{-4} & 3.8766 \\
 1/16  & 3.932609\E{-7} & 3.8585\E{-3} & 1.9480 & 3.947971\E{-7} & 3.2766\E{-5} & 3.9663 \\
 1/32  & 3.943998\E{-7} & 9.7355\E{-4} & 1.9867 & 3.947850\E{-7} & 2.0602\E{-6} & 3.9913 \\
 1/64  & 3.946879\E{-7} & 2.4395\E{-4} & 1.9967 & 3.947842\E{-7} & 1.2900\E{-7} & 3.9973 \\
 1/128 & 3.947601\E{-7} & 6.1023\E{-5} & 1.9992 & 3.947842\E{-7} & 8.3956\E{-9} & 3.9416 \\
\bottomrule
\end{tabular}
\end{table}

In case (iv), we consider the conical beam with circular cross-section. Its radius $r(x)$ linearly decreases from the value $r(0) = 0.015\,{\rm m}$ to $r(L)=0.01\,{\rm m}$.
The second moment of area about the centroidal axis varies and is given by $I(x)=\pi r^4(x)/4$.
It is an example of a beam with smoothly varying bending stiffness. 
An analytical solution is not known, so we resort to approximations by the cubic Hermite finite elements \eqref{eq:fem}. 
To find lower and upper bounds on the critical buckling load $P_1$,
we apply Corollary~\ref{co:two-sided} using the auxiliary bending problem \eqref{eq:auxweak} with piecewise constant bending stiffness.

Table~\ref{ta:decreasing} presents the computed lower and upper bounds $\lambda_{h,1}^{\text{low}}$ and $\lambda_{h,1}^{\text{up}}$ on $\lambda_1 = 4 P_1 / (E \pi)$.
The corresponding unknown relative errors $\Erellow$ and $\Erelup$ are estimated from above by $\eta_\text{rel} = (\lambda_{h,1}^{\text{up}} - \lambda_{h,1}^{\text{low}})/\lambda_{h,1}^{\text{up}}$.
The final column of Table~\ref{ta:decreasing} presents the experimental order of convergence of $\eta_\text{rel}$. 
This convergence order reduces to one due to the piecewise constant approximation of the bending stiffness in \eqref{eq:auxweak}.
Thus, the resulting lower bounds, while guaranteed, may be relatively coarse. 
Their accuracy could be substantially improved by employing the Lehmann–Goerisch method, which yields lower bounds with optimal convergence rates \cite{GoeHau1985, Vejchodsky2018}. However, a detailed description of the Lehmann–Goerisch method is technically involved and lies beyond the scope of the present paper.

\begin{table}[htb]
\centering
\caption{
Case (iv), the conical clamped beam. 
Lower and upper bounds $\lambda_{h,1}^{\text{low}}$ and $\lambda_{h,1}^{\text{up}}$ of the first eigenvalue, the corresponding estimates $\eta_\text{rel}$ of the relative error and experimental orders of convergence (EOC).}
\label{ta:decreasing}
\begin{tabular}{lcccc}
\toprule
$h$ & $\lambda_{h,1}^{\text{low}}$ & $\lambda_{h,1}^{\text{up}}$ & $\eta_\text{rel}$ & EOC \\
\midrule
1/8     &   7.782018\E{-7}  &8.889449\E{-7}  &  1.246\E{-1} &           \\
1/16    &   8.366585\E{-7}  &8.883111\E{-7}  &  5.815\E{-2} &   1.0993  \\
1/32    &   8.636850\E{-7}  &8.882674\E{-7}  &  2.767\E{-2} &   1.0711  \\
1/64    &   8.763233\E{-7}  &8.882646\E{-7}  &  1.344\E{-2} &   1.0417  \\
1/128   &   8.823860\E{-7}  &8.882644\E{-7}  &  6.618\E{-3} &   1.0225  \\
1/256   &   8.853489\E{-7}  &8.882644\E{-7}  &  3.282\E{-3} &   1.0116  \\
\bottomrule
\end{tabular}
\end{table}

The derived lower and upper bounds are not limited to the first eigenvalues associated with the fundamental buckling mode but also to higher eigenvalues corresponding to successive buckling modes. 
Figure~\ref{fi:highereigenvalues_i} presents a log–log plot of the relative errors $\Erellow$ and $\Erelup$ of the lower and upper bounds as functions of the mesh size for the first five eigenvalues in the case (i) of the uniform beam with rectangular cross-section.
We observe that approximations of lower eigenvalues are more accurate than approximations of the higher ones. However, the speeds of convergence for higher eigenvalues are the same as for the first eigenvalue. 

We also present the results for the first five eigenvalues in case (ii), where the exact eigenvalues are not known. Therefore, in Figure~\ref{fi:highereigenvalues_ii}, we plot the estimate $\eta_\text{rel}$ on both $\Erellow$ and $\Erelup$. As above, the estimate $\eta_\text{rel}$ is dominated by the error of the lower bound, and we observe its second-order convergence rate for all eigenvalues. We also observe more accurate approximations of lower eigenvalues.
Corresponding curves for cases (iii) and (iv) are similar to those in Figures~\ref{fi:highereigenvalues_i} and \ref{fi:highereigenvalues_ii}, respectively, and we do not present them.

\begin{figure}[tb]
    \centering
    \includegraphics[width=\textwidth]{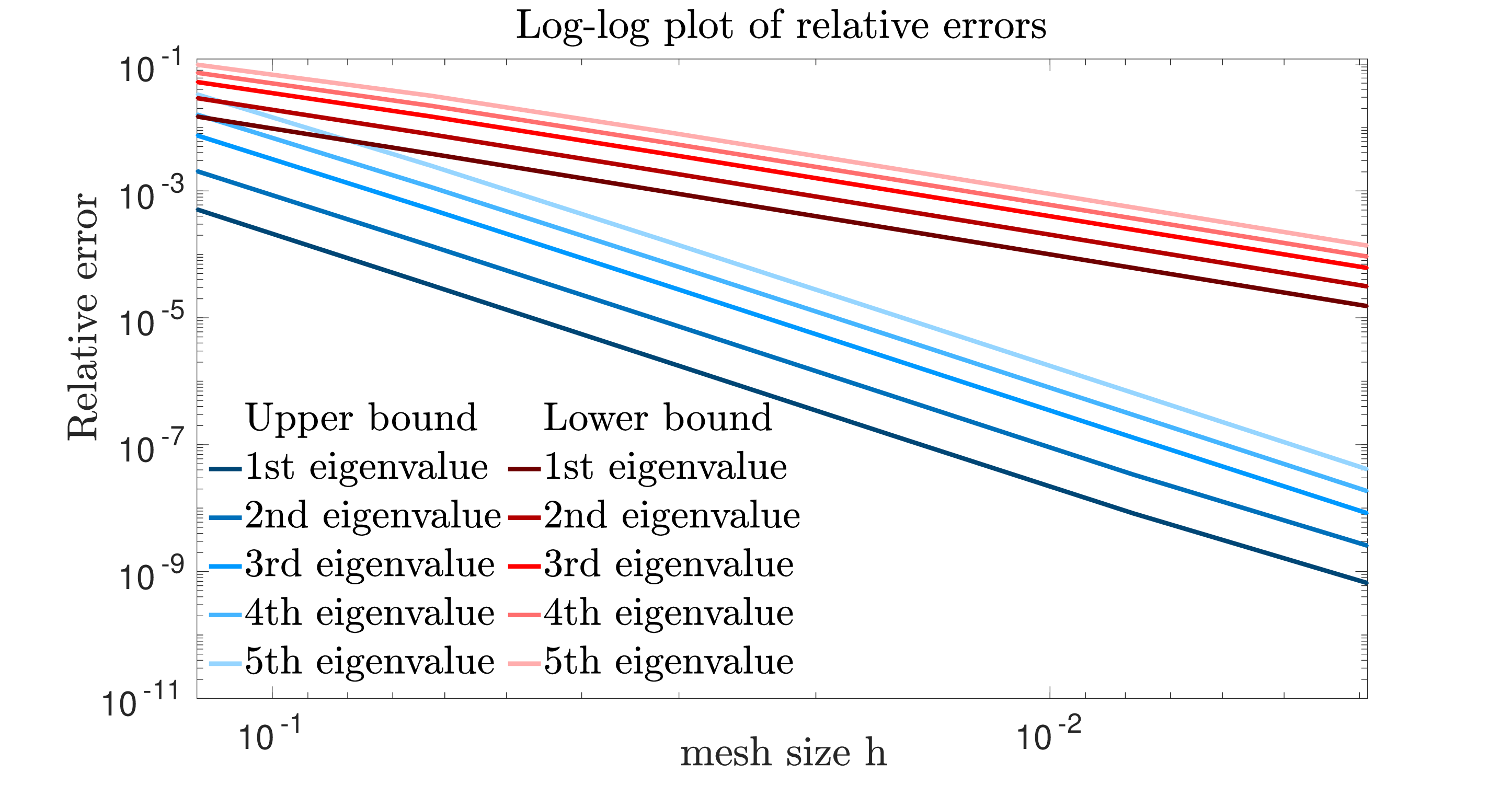} 
    \caption{Case (i), uniform beam with rectangular cross-section. Log-log plot of relative errors $\Erellow$ (in red) and $\Erelup$ (in blue) for the first five eigenvalues.
    }
    \label{fi:highereigenvalues_i}
\end{figure}

\begin{figure}[htb] 
    \centering
    \includegraphics[width=\textwidth]{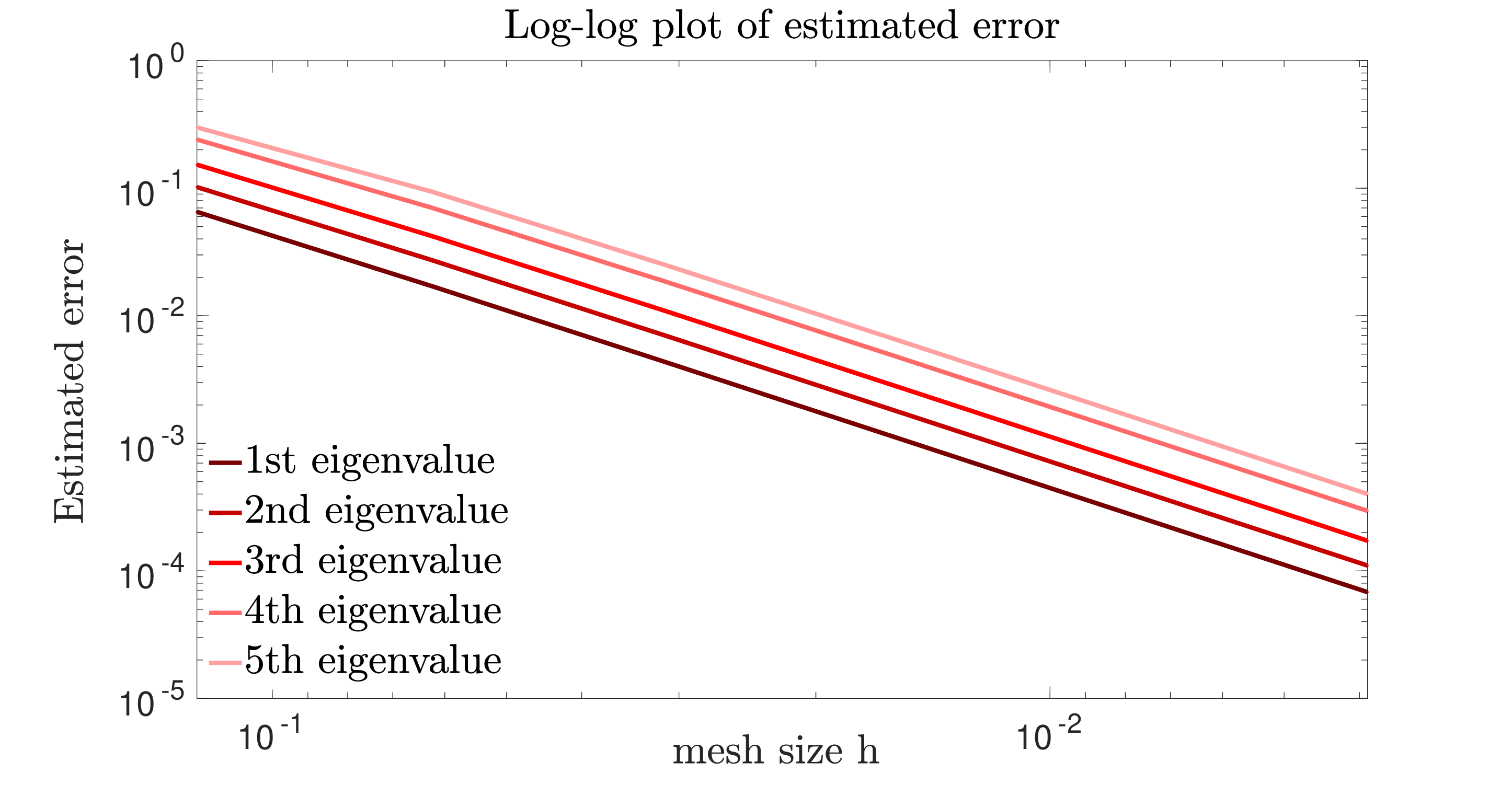} 
    \caption{Case (ii), a stepped beam with rectangular cross-section. Log-log plot of estimates of the relative error $\eta_\text{rel}$ for the first five eigenvalues.
    }
    \label{fi:highereigenvalues_ii}
\end{figure}

\section{Conclusions and Future Work}
\label{eq:conclusion}	

We have presented a method to obtain two-sided bounds for eigenvalues in the Euler--Bernoulli beam buckling problem. Using a combination of the cubic Hermite elements and an interpolation inequality, we derived a formula for a guaranteed lower bound for all eigenvalues. Especially interesting is the application to the first eigenvalue corresponding to the critical buckling load. The method is especially efficient and easy to apply in the case of beams with piecewise constant bending stiffness, such as beams with stepped thickness.
Since the critical load of the nonlinear Gao beam model coincides in this case with the first eigenvalue of the classical Euler–Bernoulli beam, the derived bounds can be directly used for the physically more relevant Gao beam model.

In the general case of varying bending stiffness $EI$ along the length of the beam, we compute guaranteed lower bounds on eigenvalues by solving an auxiliary problem with piecewise constant bending stiffness.

Although the derivation was presented for clamped boundary conditions, the extension to other types of boundary conditions is straightforward. 
The mesh-dependent constant $C_h$ given in \eqref{eq:Ch} remains the same for all boundary conditions. Indeed, the proofs of crucial Lemmata~\ref{le:interpestim} and \ref{le:auIhu} depend solely on properties of the interpolation operator $\cI_h u$ and, in particular, on the facts that the function $w = u - \mathcal{I}_h u$ vanishes together with its derivatives at all points of the finite element partition. Consequently, $w|_{K_k} \in H_0^2(K_k)$ and $w'|_{K_k} \in H_0^1(K_k)$ for all element $K_k = [x_k,x_{k+1}]$, $k=1,2,\dots,N$.
Of course, the choice of boundary conditions affects the value of the first eigenvalue of the associated beam buckling problem, see e.g., \cite{Eisley2011} for more details.
 
Standard upper bounds obtained by the Rayleigh--Ritz method have the convergence of order four.
Numerical experiments demonstrate that the derived lower bounds converge quadratically to the true eigenvalue in the case of piecewise constant bending stiffness and linearly for generally varying coefficients. Although not optimally convergent, these lower bounds are sufficiently accurate to provide valuable information about the buckling loads, which may guide engineering decisions.
	
The key advantages of this approach are its simplicity and reliance on standard computational outputs (cubic Hermite elements and known constants). This makes it attractive for integration into finite element software to provide two-sided error bounds for eigenvalue computations automatically.
	
If needed, lower bounds with an optimal speed of convergence can be obtained by combining the proposed approach with the Lehmann--Goerish method \cite{GoeHau1985,Lehmann1949,Lehmann1950}. The Lehmann--Goerish method provides highly accurate and optimally convergent lower bounds on eigenvalues, provided an \emph{a priori} known, perhaps rough, lower bound on some eigenvalue is available. The needed \emph{a priori} lower bound can be easily computed by using Theorem~\ref{th:two-sided}. The advantage of such an approach would be optimally convergent lower bounds even for higher-order approximations. 
On top of that, the Lehmann--Goerish method provides a natural means to steer the automatic mesh adaptation. However, it is more computationally demanding, and its implementation is sophisticated.

Generalization of this approach to plate buckling is well possible, see \cite{CarGal2014} for an application of these ideas to a biharmonic problem in a two-dimensional domain.

In conclusion, the combination of mathematical analysis and numerical verification provided in this paper offers a reliable way to bracket eigenvalues for the Euler--Bernoulli beam. We anticipate that the approach can be a useful addition to the toolkit of computational mechanics, enhancing confidence in stability predictions for beams and related structural elements.

\backmatter





\bmhead{Acknowledgements}
Jitka Machalová gratefully acknowledges the support of the grant IGA\_PrF\_2026\_018, Mathematical models. 
Jana Burkotová gratefully acknowledges the support by the Ministry of Education, Youth and Sports of the Czech Republic (MEYS CR) under the project Hydrodynamic design of pumps CZ.02.1.01/0.0/0.0/17\_049/0008408. 
Tom\'a\v{s} Vejchodský acknowledges the support of the Czech Science Foundation, grant no. GA23-06159S, and the institutional funding of the Czech Academy of Sciences, RVO 67985840.

\bibliography{vejchod_aee}

\end{document}